\documentclass[amstex,12pt]{article}%
\usepackage{graphicx}
\usepackage{amsfonts}
\usepackage{amssymb}
\usepackage{amsmath}
\usepackage{amsthm}%
\newtheorem{thm}{Theorem}
\newtheorem{lemma}{Lemma}
\input{Macros.tex}
\df\P{{\mathcal P}}
\df\Q{{\mathcal Q}}
\df\D{{\mathcal D}}
\df\X{{\mathcal X}}
\df\G{{\Gamma}}
\df\g{{\gamma}}
\df\l{{\lambda}}
\df\L{{\Lambda}}
\df\k{{\kappa}}
\df\o{{\omega}}
\df\O{{\Omega}}
\df\r{{\rho}}
\df\s{{\sigma}}
\df\d{{\delta}}
\df\e{{\varepsilon}}
\df\a{{\alpha}}
\df\b{{\beta}}
\df\t{{\theta}}
\df\R{{\mathbb{R}}}
\df\N{{\mathbb{N}}}
\df\Z{{\mathbb{Z}}}
\df\s{{\mathbb{S}}}
\begin{document}

\author{Jo\"{e}l Rouyer
\and Costin V\^{\i}lcu}
\title{The connected components\\of the space of Alexandrov surfaces}
\maketitle

\begin{abstract}
Denote by $\mathcal{A}(\kappa)$ the set of all compact Alexandrov surfaces
with curvature bounded below by $\kappa$ without boundary, endowed with the
topology induced by the Gromov-Hausdorff metric. We determine the connected
components of $\mathcal{A}(\kappa)$ and of its closure.

\end{abstract}

{\small Math. Subj. Classification (2010): 53C45}

{\small Key words and phrases: space of Alexandrov surfaces}


\section{Introduction and results}

In this note, by an \textit{Alexandrov surface} we understand a compact
$2$-dimen\-sio\-nal Alexandrov space with curvature bounded below by $\kappa$,
without boundary. Roughly speaking, an Alexandrov surface is a closed
topological surface endowed with an intrinsic geodesic distance satisfying
Toponogov's angle comparison condition. See \cite{BGP} or \cite{Shiohama92}
for definitions and basic facts about such spaces.

Denote by $\mathcal{A}(\kappa)$ the set of all Alexandrov surfaces. Endowed
with the Gromov-Hausdorff metric $d_{GH}$, $\mathcal{A}(\kappa)$ becomes a Baire
space in which Riemannian surfaces form a dense subset \cite{IRV2}.

Perelman's stability theorem (see \cite{Kapo1}, \cite{Per1}) states, in our
case, that close Alexandrov surfaces are homeomorphic, so Alexandrov surfaces
with different topology are in different connected components of
$\mathcal{A}(\kappa)$. Here we show that homeomorphic Alexandrov surfaces are
in the same component of $\mathcal{A}(\kappa)$.

Let $\mathcal{A}(\kappa,\chi,o)$ denote the set of all surfaces in
$\mathcal{A}(\kappa)$ of Euler-Poincar\'{e} characteristic $\chi$ and
orientability $o$, where $o=1$ if the surface is orientable and
$o=-1$ otherwise.

\begin{thm}
\label{1} If non-empty, $\mathcal{A}\left(  \kappa,\chi,o\right) $ is a
connected component of $\mathcal{A}\left(  \kappa\right) $, for $\kappa
\in\mathbb{R}$, $\chi\leq2$ and $o=\pm1$.
\end{thm}

A special motivation for this result comes from the study of \textit{most} (in
the sense of Baire category) Alexandrov surfaces. For example, we prove in
\cite{RV2} that most Alexandrov surfaces have either infinitely many simple
closed geodesics, or no such geodesic, depending on the value of $\kappa$ and
the connected component of $\mathcal{A}(\kappa)$ to which they belong.
Moreover, for descriptions of most Alexandrov surfaces given in \cite{A-Z} and
\cite{IRV2}, one has to exclude from the whole space $\mathcal{A}(0)$ its
components consisting of flat surfaces.

\medskip

Denote by $\bar{\mathcal{A}}(\kappa)$ (respectively $\bar{\mathcal{A}}%
(\kappa,\chi,o)$) the closure with respect to $d_{GH}$ of $\mathcal{A}(\kappa)$ (respectively
$\mathcal{A}\left(\kappa,\chi,o\right)$) in the space of all compact
metric spaces. Using Theorem \ref{1}, we can also
give the connected components of $\bar{\mathcal{A}}(\kappa)$.

\begin{thm}
\label{2} If $\kappa\geq0$, $\bar{\mathcal{A}}(\kappa)$ is connected. If
$\kappa<0$, the connected components of $\bar{\mathcal{A}}(\kappa)$ are
$\bigcup_{\chi\geq0,o=\pm1}\bar{\mathcal{A}}\left(\kappa,\chi,o\right)$ and $\mathcal{A}(\kappa,\chi,o)$ ($\chi=-1,-2,\ldots$, $o=\pm1$).
\end{thm}


\section{Proofs}

Perelman's stability theorem can be found, for example, in \cite{Per1} or
\cite{Kapo1}; we only need a particular form of it.

\begin{lemma}
\label{stability} Each Alexandrov surface $A$ has a neighbourhood in
$\mathcal{A}\left(  \kappa\right)  $ whose elements are all homeomorphic to
$A$.
\end{lemma}

Let $\mathbb{M}_{\kappa}^{d}$ stand for the simply-connected and complete
Riemannian manifold of dimension $d$ and constant curvature $\kappa$.

Denote by $\mathcal{R} (\kappa)$ the set of all closed Riemannian surfaces
with Gauss curvature at least $\kappa$, and by $\mathcal{P}(\kappa)$ the set
of all $\kappa$-polyhedra. Recall that a $\kappa$\emph{-polyhedron} is an
Alexandrov surface obtained by naturally gluing finitely many geodesic
polygons from $\mathbb{M}_{\kappa}^{2}$.

\medskip

A formal proof for the following result can be found in \cite{IRV2}.

\begin{lemma}
\label{dense} The sets $\mathcal{R}(\kappa)$ and $\mathcal{P}(\kappa)$ are
dense in $\mathcal{A}(\kappa)$.
\end{lemma}

A \emph{convex surface} in $\mathbb{M}_{\kappa}^{3}$ is the boundary of a
compact convex subset of $\mathbb{M}_{\kappa}^{3}$ with non-empty interior.
Such a surface is endowed with the so-called intrinsic metric: the distance
between two points is the length (measured with the metric of $\mathbb{M}%
_{\kappa}^{3}$) of a shortest curve joining them and lying on the surface.

\begin{lemma}
\label{real} \cit{al} Every convex surface in $\mathbb{M}_{\kappa}^{3}$
belongs to $\mathcal{A}\left(  \kappa,2,1\right) $. Conversely, every surface
$A\in\mathcal{A}\left(  \kappa,2,1\right) $ is isometric to some convex
surface in $\mathbb{M}_{\kappa}^{3}$.
\end{lemma}

In order to settled the case of $\mathcal{A}\left(  0,0,o\right) $, we
need the following lemma.

\begin{lemma}
\label{LFlat} $\mathcal{A}\left(  0,0,1\right) $ contains only flat tori, and
$\mathcal{A}\left( 0,0,-1\right) $ contains only flat Klein bottles.
\end{lemma}

\begin{proof} Recall that geodesic triangulations with arbitrarily small triangle 
exist for any Alexandrov surface \cite{al}.

Consider $A\in\mathcal{A}\left(0,0,1\right)$ and a geodesic triangulation
$T=\left\{ \Delta_{i}\right\}$ of $A$. For each $\Delta_{i}$,
consider a comparison triangle $\tilde{\Delta}_{i}$ (\ie, a triangle with the
same edge lengths) in $\mathbb{M}^2_{0}$. 
Glue together the triangles $\tilde{\Delta}_{i}$ to obtain a surface $P$, 
in the same way the triangles $\Delta_{i}$ are glued together to compose $A$.
By the definition of Alexandrov surfaces, the angles of $\tilde{\Delta}_{i}$ 
are lower than or equal to the angles of $\Delta_{i}$.
It follows that the total angles $\theta_{1}$, \ldots, $\theta_{n}$ of $P$ around its
(combinatorial) vertices are at most $2\pi$, hence $P$ is a $0$-polyhedron.
By the Gauss-Bonnet formula for polyhedra,
\[
0=2\pi\chi=\sum_{i=1}^{n}\left(  2\pi-\theta_{_{i}}\right)  \text{,}%
\]
whence $\theta_{i}=2\pi$ and $P$ is indeed a flat torus.

Now consider a sequence of finer and finer triangulations $T_{m}$ of $A$ and denote by $P_{m}$ the corresponding flat tori ($m \in \mathbb{N}$). A result of Alexandrov and Zalgaller 
(Theorem 10 in \cite[p. 90]{AZ}) assures that $P_{m}$ converges to $A$, which is therefore flat.

The same argument holds for $\mathcal{A}\left(0,0,-1\right)$.
\end{proof}

Now we are in a position to prove Theorem \ref{1}.

Notice that, for $\kappa^{\prime}>\kappa$, $\mathcal{A}\left(\kappa^{\prime}\right)$ 
is a nowhere dense subset of $\mathcal{A}\left(\kappa\right)$; indeed, $\mathcal{A}\left(\kappa^{\prime}\right)$ is closed and its complement contains the $\kappa$-polyhedra, which are dense in $\mathcal{A}\left(\kappa\right)$. Therefore, there is no direct relationship between the connected components of $\mathcal{A}\left(\kappa\right)$ and those of $\mathcal{A}\left(\kappa^{\prime}\right)$.

\begin{proof}
[Proof of Theorem \ref{1}]By Lemma \ref{stability}, each set $\mathcal{A}%
\left( \kappa,\chi,o\right) $ is open in $\mathcal{A}\left(
\kappa\right) $, so we just need to prove that it is connected.

Each Alexandrov surface $A$ is in particular a metric space. Multiplying all
distances in $A\in\mathcal{A}(\kappa)$ with the same constant $\delta>0$
provides another Alexandrov surface, denoted by $\delta A$, which belongs to
$\mathcal{A}(\frac{\kappa}{\delta^{2}})$. Moreover, it is easy to see that for
any metric spaces $M$, $N$ we have $d_{GH}\left(  \delta M,\delta N\right)
=\delta d\left( M,N\right)$. So there is a natural homothety between
$\mathcal{A}(\kappa)$ and $\mathcal{A}(\frac{\kappa}{\delta^{2}})$, and
therefore we may assume that
\[
\kappa\in\{-1,0,1\}.
\]

We consider several cases.

\medskip

\textit{Case 1.} The sets $\mathcal{A}(-1,\chi,o)$ are connected in
$\mathcal{A}(-1)$.

Choose $A_{0}$, $A_{1}\in\mathcal{A}(-1,\chi,o)\cap\mathcal{R}(-1)$.
There exist a differentiable surface $S$ of Euler-Poincar\'{e} characteristic $\chi$
and orientability $o$, and Riemannian metrics $g_{0}$, $g_{1}$ on $S$
such that $A_{i}$ is isometric to $\left(  S,g_{i}\right)  $ ($i=0$,$1$). For
$\lambda\in\lbrack0,1]$ we set
\[
\tilde{g}_{\lambda}=\lambda g_{1}+(1-\lambda)g_{0}\text{.}%
\]
Denote by $\kappa_{\lambda}$ the minimal value of the Gauss curvature of
$\tilde{g}_{\lambda}$, and define the Riemannian metric $g_{\lambda}$ on $S$
by
\[
g_{\lambda}=\left\{
\begin{array}
[c]{l}%
\tilde{g}_{\lambda}\\
\frac{\tilde{g}_{\lambda}}{\sqrt{-\kappa_{\lambda}}}%
\end{array}
\right.
\begin{array}
[c]{l}%
\text{if }\kappa_{\lambda}\geq-1\text{,}\\
\text{if }\kappa_{\lambda}<-1\text{.}%
\end{array}
\]
A straightforward computation shows that the Gauss curvature $K_{\lambda}$ of
$g_{\lambda}$ verifies $K_{\lambda}\geq-1$.

Denote by $\gamma$ the (obviously continuous) canonical map from the set of Riemannian structures on $S$ to $\mathcal{A}(-1,\chi,o)$, which maps $g$ to $\left( S,g\right)$. Then $A_{\lambda}\overset{\mathrm{def}}{=}\gamma\left( g_{\lambda}\right)$ defines a path from $A_{0}$ to $A_{1}$. 
Hence $\mathcal{A}(-1,\chi,o)\cap\mathcal{R}(-1)$ is connected and, by the
density of $\mathcal{R}(-1)$, so is $\mathcal{A}(-1,\chi,o)$.

\medskip

Next we treat the connected components of $\mathcal{A}(0)$.

\medskip

\textit{Case 2.} The sets $\mathcal{A}\left(  0,0,1\right) $ and
$\mathcal{A}\left( 0,0,-1\right) $ are connected in $\mathcal{A}(0)$.

By Lemma \ref{LFlat}, the set $\mathcal{A}\left(  0,0,1\right)  $ contains
only flat tori, hence it is conti\-nu\-ously para\-me\-trized by the parameters
describing the fundamental domains. Similarly for $\mathcal{A}\left(
0,0,-1\right)  $, which consists of flat Klein bottles.

\medskip

\textit{Case 3.} The set $\mathcal{A}(0,2,1)$ is connected in $\mathcal{A}(0)$.

Denote by $\mathcal{S}$ the space of all convex surfaces in ${\mathbb{R}}^{3}$, endowed with the Pompeiu-Hausdorff metric. Lemma \ref{real} shows that any surface $A\in\mathcal{A}(0,2,1)$ can be realized as a convex surface in ${\mathbb{R}}^{3}$.

Given two convex surfaces $S_{0}$, $S_{1}$, define for $\lambda\in
\lbrack0,1]$
\begin{equation}
S_{\lambda}=\partial\left(  \lambda\mathrm{conv}\left(  S_{1}\right)
+(1-\lambda)\mathrm{conv}\left(  S_{0}\right)  \right)  \text{,}\label{Sl}%
\end{equation}
where $\partial C$ stands for the boundary of $C$, $\mathrm{conv}\left(
S\right)  $ for the convex hull of $S$, and $+$ for the Minkowski sum. Then
$S_{\lambda}\in\mathcal{S}$ and we have a path in $\mathcal{S}$ joining
$S_{0}$ to $S_{1}$. Since the canonical map $\sigma$ from $\mathcal{S}$ to
$\mathcal{A}(0,2,1)$ is continuous \cite[Theorem 1 in Chapter 4]{al}, we
obtain a path in $\mathcal{A}(0,2,1)$.

\medskip

\textit{Case 4.} The set $\mathcal{A}(0,1,-1)$ is connected in $\mathcal{A}%
(0)$.

Consider surfaces $A_{0}$, $A_{1}$ in $\mathcal{A}(0,1,-1)$ as quotients of
centrally-symme\-tric convex surfaces $S_{0}$, $S_{1}$ via antipodal
identification, $A_{i}=\sigma\left(  S_{i}\right)  /{\mathbb{Z}}_{2}$
($i=0,1$). Then the surface $S_{\lambda}$ defined by (\ref{Sl}) is also
centrally-symmetric, and therefore $A_{\lambda}=\sigma\left(  S_{\lambda
}\right)  /{\mathbb{Z}}_{2}$ defines a path in $\mathcal{A}(0,1,-1)$ from $A_{0}$
to $A_{1}$.

\medskip

We finally treat the two connected components of $\mathcal{A}(1)$.

\medskip

\textit{Case 5.} The set $\mathcal{A}(1,2,1)$ is connected in $\mathcal{A}(1)$.

Consider in ${\mathbb{R}}^{4}$ the subspace ${\mathbb{R}}^{3}={\mathbb{R}}%
^{3}\times\{0\}$, and the open half-sphere $H$ of center $c=(0,0,0,1)$ and
radius $1$ included in ${\mathbb{R}}^{3}\times\lbrack0,1[$.

Let $q:{\mathbb{R}}^{3}\rightarrow H$ be the homeomorphism associating to each
$x\in{\mathbb{R}}^{3}$ the intersection point of the segment $[xc]$ with $H$.
Clearly, $q$ maps segments of ${\mathbb{R}}^{3}$ to geodesic segments of $H$,
and thus it maps bijectively convex sets in ${\mathbb{R}}^{3}$ to convex sets in $H$.
Denote $\mathcal{S}_{H}$ the set of convex surfaces in $H$. We can define
$Q:\mathcal{S}\rightarrow\mathcal{S}_{H}$ by $Q(S)\overset{\mathrm{def}}%
{=}q(S)$. Hence $\mathcal{S}_{H}$ is homeomorphic to $\mathcal{S}$, which is
connected by Case (4).

Consider now two surfaces $A_{0}$, $A_{1}\in\mathcal{A}(1,2,1)$ and choose
\[
\mu< \min\left\{ \frac{\pi}{2 \mathrm{\mathrm{diam}}\left( A_{0}\right)}, \frac{\pi}{2 \mathrm{\mathrm{diam}}\left( A_{1}\right) },1\right\} \text{.}%
\]
Obviously, $A_{i}$ is path-connected to $\mu A_{i}$ in $\mathcal{A}(1,2,1)$,
and the diameter of $\mu A_{i}$ is less than $\pi/2$ ($i=0$, $1$). 
By Lemma \ref{real}, $\mu A_{i}$ is isometric to a surface $S_{i}$ in $\mathbb{M}_1^{3}$; 
moreover, the smallness of $\mu$ easily implies that $S_{i}$ is isometric to a surface in $\mathcal{S}_{H}$, 
and $\mathcal{S}_{H}$ is connected.

\medskip

\textit{Case 6.} The set $\mathcal{A}(1,1,-1)$ is connected in $\mathcal{A}%
(1)$.

This follows directly from the previous argument, because the universal
covering of any surface $\tilde{A}\in\mathcal{A}(1,1,-1)$ is a surface
$A\in\mathcal{A}(1,2,1)$ endowed with an isometric involution without fixed
points, $\tilde{A}=A/{\mathbb{Z}}_{2}$.

\medskip

The proof of Theorem \ref{1} is complete.
\end{proof}

Recall that the $2$-dimensional Hausdorff measure $\mu(A)$ is always finite
and positive for $A \in\mathcal{A}(\kappa)$. The following result is Corollary
10.10.11 in \cite[p. 401]{bbi}, stated in our framework.

\begin{lemma}
\label{no_collaps} Let $A_{n} \in\mathcal{A}(\kappa)$ converge to a compact
space $X$. Then $\dim(X)<2$ if and only if $\mu(A_{n}) \to0$.
\end{lemma}

\begin{proof}
[Proof of Theorem \ref{2}]We may assume, as in the proof of Theorem \ref{1},
that $\kappa\in\{-1,0,1\}$.

To prove that $\bar{\mathcal{A}}(\kappa)$ is connected for $\kappa\geq0$, it
suffices to show that the space consisting of a single point belongs to the
closure of any connected component of $\mathcal{A}(\kappa)$. This is indeed
the case, because for any $A\in\mathcal{A}(\kappa,\chi,o)$ and
$0<\delta\leq1$ we have $\delta A\in\mathcal{A}(\kappa,\chi,o)$, and
$\lim_{\delta\rightarrow0}\delta A$ is a point.

This also implies that
\[
\bigcup_{\substack{o=\pm1\\\chi=0,1,2}}\bar{\mathcal{A}}(-1,\chi,o)
\]
is connected.

Consider now $A\in\mathcal{A}\left(  -1,\chi,o\right)$ with $\chi<0$.
Let $\omega$ be the curvature measure on $A$ (see \cite{AZ}). 
Y. Machigashira \cite{m} proved that $\omega\geq\kappa\mu$ holds for any Alexandrov surface 
of curvature bounded below by $\kappa$. Therefore, by a variant of the Gauss-Bonnet theorem,
\[
2\pi\chi=\omega(A)\geq\kappa\mu\left(A\right)=-\mu\left( A\right)
\text{,}%
\]
hence $\mu(A)\geq2\pi|\chi|$. Lemma \ref{no_collaps} shows now that
$\mathcal{A}(-1,\chi,o)$ is closed in the space of all compact metric
spaces ($o=\pm1,\chi<0$).
\end{proof}


\bigskip

\noindent\textbf{Acknowledgement.} The authors were supported by the grant
PN-II-ID-PCE-2011-3-0533 of the Romanian National Authority for Scientific
Research, CNCS-UEFISCDI.


\bigskip

J\"oel Rouyer

\noindent{\small Institute of Mathematics ``Simion Stoilow'' of the Romanian
Academy, \newline P.O. Box 1-764, Bucharest 70700, ROMANIA \newline
Joel.Rouyer@ymail.com, Joel.Rouyer@imar.ro}

\medskip

Costin V\^{\i}lcu

\noindent{\small Institute of Mathematics ``Simion Stoilow'' of the Romanian
Academy, \newline P.O. Box 1-764, Bucharest 70700, ROMANIA \newline
Costin.Vilcu@imar.ro}


\begin{thebibliography}{99}                                                                                               %


\bibitem {A-Z}K. Adiprasito and T. Zamfirescu, \textit{Few Alexandrov spaces
are Riemannian}, submitted, 2012

\bibitem {al}A. D. Alexandrov, \textsl{Die innere Geometrie der konvexen
Fl\"{a}chen}, Akademie-Verlag, Berlin, 1955

\bibitem {AZ}A.D. Aleksandrov and V.A. Zalgaller, \textsl{Intrinsic geometry
of surfaces}, Transl. Math. Monographs, Providence, RI, Amer. Math. Soc., 1967.

\bibitem {bbi}D. Burago, Y. Burago and S. Ivanov, \textsl{A course in metric
geometry}, Amer. Math. Soc., Providence, RI, 2001

\bibitem {BGP}Yu. Burago, M.~Gromov, and G.~Perel'man, \textit{A. D.
Alexandrov spaces with curvature bounded below}., Russ. Math. Surv.
\textbf{47} (1992), no.~2, 1--58 (English. Russian original)

\bibitem {IRV2}J. Itoh, J. Rouyer and C. V\^{\i}lcu, \textit{Moderate
smoothness of most Alexandrov surfaces}, arXiv:1308.3862 [math.MG]

\bibitem {Kapo1}V. Kapovitch, \textit{Perelman's stability theorem}, J.
Cheeger et al. (eds.), Metric and comparison geometry. International Press.
Surveys in Differential Geometry 11 (2007), 103-136

\bibitem {m}Y. Machigashira, \textit{The Gaussian curvature of Alexandrov
surfaces}, J. Math. Soc. Japan \textbf{50} (1998), 859-878

\bibitem {Per1}G. Perel'man, \textit{A.D. Alexandrov spaces with curvatures
bounded from below II}, preprint 1991

\bibitem {RV2}J. Rouyer and C. V\^{\i}lcu, \textit{Simple closed geodesics on
most Alexandrov surfaces}, manuscript

\bibitem {Shiohama92}K. Shiohama, \textsl{An introduction to the geometry of
Alexandrov spaces}, Lecture Notes Serie, Seoul National University, 1992
\end{thebibliography}
\end{document}